\documentclass[a4paper,12pt,twoside]{amsart}
\usepackage[utf8]{inputenc}
\usepackage{latexsym}
\usepackage[pdftex]{graphicx}
\usepackage{hyperref}
\usepackage{amsmath,amssymb, amsthm, amscd}
\usepackage{multicol}
\usepackage[marginratio=1:1,height=597pt,width=460pt,tmargin=117pt]{geometry}
\usepackage[all]{xy}
\usepackage{etoolbox}
\usepackage{enumitem}
\usepackage{tikz-cd}
\expandafter\patchcmd\csname\string\proof\endcsname
  {\normalparindent}{0pt }{}{}
\makeatletter
\patchcmd{\@thm}{\thm@headfont{\scshape}}{\thm@headfont{\scshape\bfseries}}{}{}
\patchcmd{\@thm}{\thm@notefont{\fontseries\mddefault\upshape}}{}{}{}
\makeatother
\makeatletter
\patchcmd\@thm
  {\let\thm@indent\indent}{\let\thm@indent\indent}%
  {}{}
\makeatother
\newtheorem{theorem}[equation]{Theorem}

\theoremstyle{definition}

\theoremstyle{definition}

\theoremstyle{remark}

\numberwithin{equation}{section}
\newcommand{\iso}{\xrightarrow{
   \,\smash{\raisebox{-0.50ex}{\ensuremath{\scriptstyle\sim}}}\,}}
\title
[Non-admissible irreducible representations]
{On non-admissible irreducible modulo $p$ representations of $GL_{2}(\mathbb{Q}_{p^{2}})$}
\author{Eknath Ghate}
\address{School of Mathematics, Tata Institute of Fundamental Research \\ Homi Bhabha Road, Mumbai - 400005, India.}
\email{eghate@math.tifr.res.in}
\author{Mihir Sheth}
\address{School of Mathematics, Tata Institute of Fundamental Research \\ Homi Bhabha Road, Mumbai - 400005, India.}
\email{mihir@math.tifr.res.in}
\begin{document}
\maketitle
\begin{abstract}
We use a Diamond diagram attached to a 2-dimensional reducible split mod $p$ Galois representation of $\mathrm{Gal}(\overline{\mathbb{Q}_{p}}/\mathbb{Q}_{p^{2}})$ to 
construct a non-admissible smooth irreducible mod $p$ representation of $GL_{2}(\mathbb{Q}_{p^{2}})$ following the approach of Daniel Le. 
\end{abstract}
\tableofcontents
\section{Introduction}

Let $p$ be a prime number, $\mathbb{Q}_{p}$ be the field of $p$-adic numbers, and $\overline{\mathbb{F}_{p}}$ be the algebraic closure of the finite field $\mathbb{F}_{p}$ of cardinality $p$. The study of the admissibility of smooth irreducible representations of connected reductive $p$-adic groups goes back 
to Harish-Chandra (\cite{HC70}). Building upon his work, Jacquet proved that every such representation over the field of complex numbers is admissible (\cite{Jac75}, see also Bernstein \cite{Ber74}). This result was 
extended by Vign\'{e}ras to smooth irreducible representations over any algebraically closed field of characteristic not equal to $p$ (\cite{Vig96}). It is also known that every 
smooth irreducible representation of $GL_{2}(\mathbb{Q}_{p})$ over $\overline{\mathbb{F}_{p}}$ is admissible (see Berger \cite{ber12}). 
Let $\mathbb{Q}_{p^2}$ be the unramified extension of $\mathbb{Q}_{p}$ of degree $2$. In this paper, we establish 
the existence of a \emph{non-admissible} smooth irreducible $\overline{\mathbb{F}_{p}}$-linear  representation of $GL_{2}(\mathbb{Q}_{p^{2}})$, for $p > 2$, following the 
approach of Daniel Le (\cite{le19}). Our result supports the viewpoint of Breuil and Pa\v sk$\bar{\mathrm{u}}$nas that the mod $p$ (and $p$-adic) representation 
theory of $GL_{2}(\mathbb{Q}_{p^{2}})$ is more complicated than that of $GL_{2}(\mathbb{Q}_{p})$ (\cite{bp12}, see also Schraen \cite{Sch15}). 

Let $G=GL_{2}(\mathbb{Q}_{p^2})$, $K=GL_{2}(\mathbb{Z}_{p^2})$, and $\Gamma =GL_{2}(\mathbb{F}_{p^2})$, where $\mathbb{Z}_{p^2}$ is the ring of integers
of $\mathbb{Q}_{p^{2}}$ with residue field $\mathbb{F}_{p^2}$. Fix an embedding $\mathbb{F}_{p^{2}}\hookrightarrow\overline{\mathbb{F}_{p}}$. Let $I$ and $I_{1}$ 
denote the Iwahori and the pro-$p$ Iwahori subgroups of $K$ respectively, and $K_{1}$ denote the first principal congruence subgroup of $K$. Write $N$ for 
the normalizer of $I$ (and of $I_{1}$) in $G$. As a group, $N$ is generated by $I$, the center $Z$ of $G$, and by the element $\Pi= \left( \begin{smallmatrix}
0 & 1 \\ p & 0 \end{smallmatrix} \right)$. All representations considered in this paper from now on are over $\overline{\mathbb{F}_{p}}$-vector spaces. For a character $\chi$ of 
$I$, $\chi^{s}$ denotes its $\Pi$-conjugate sending $g$ in $I$ to $\chi(\Pi g\Pi^{-1})$.

A weight  is a smooth irreducible representation of $K$. The $K$-action on such a representation factors through $\Gamma$ and thus any weight is 
described by a 2-tuple $(r_{0},r_{1})\otimes\mathrm{det}^{m}:=\mathrm{Sym}^{r_{0}}\overline{\mathbb{F}_{p}}^2\otimes(\mathrm{Sym}^{r_{1}}\overline{\mathbb{F}_{p}}^2)^{\mathrm{Frob}}\otimes\mathrm{det}^{m}$ of integers with $0\leq r_{0},r_{1}\leq p-1$ together with a determinant twist for some $0\leq m < p^{2}-1$ 
(\cite{bre07}, Lemma 2.16 and Proposition 2.17). Given a weight $\sigma$, its subspace $\sigma^{I_{1}}$ of $I_{1}$-invariants has dimension 1. 
If $\chi_{\sigma}$ denotes the corresponding smooth character of $I$, then there exists a unique weight $\sigma^{s}$ such that $\chi_{\sigma^{s}}=\chi_{\sigma}^{s}$ 
(\cite{pas04}, Theorem 3.1.1).

A \emph{basic 0-diagram} is a triplet $(D_{0},D_{1},r)$ consisting of a smooth $KZ$-representation $D_{0}$, a smooth $N$-representation $D_{1}$ and an
 $IZ$-equivariant isomorphism $r:$ $D_{1}\iso D_{0}^{I_{1}}$ with the trivial action of $p$ on $D_{0}$ and $D_{1}$. Given such a diagram such that 
$D_{0}^{K_{1}}$ has finite dimension, the smooth injective $K$-envelope $\mathrm{inj}_{K}D_{0}$ admits a non-canonical $N$-action which glues 
together with the $K$-action to give a smooth $G$-action on $\mathrm{inj}_{K}D_{0}$ (\cite{bp12}, Theorem 9.8). The $G$-subrepresentation 
of $\mathrm{inj}_{K}D_{0}$ generated by $D_{0}$ is smooth admissible and its $K$-socle equals the $K$-socle $\mathrm{soc}_{K}D_{0}$ of $D_{0}$. 

From now on, assume that $p$ is odd. 
Let $\rho:\mathrm{Gal}(\overline{\mathbb{Q}_{p}}/\mathbb{Q}_{p^{2}})\rightarrow GL_{2}(\overline{\mathbb{F}_{p}})$ be a continuous generic Galois 
representation such that $p$ acts trivially on its determinant and $\mathcal{D}(\rho)$ be the set of weights, called \emph{Diamond weights}, associated to $\rho$ as 
described in \cite{bp12}, Section 11. Breuil and Pa\v sk$\bar{\mathrm{u}}$nas attach a family of basic $0$-diagrams $(D_{0}(\rho), D_{1}(\rho),r)$, 
called \emph{Diamond diagrams}, to $\rho$ such that $\mathrm{soc}_{K}D_{0}(\rho)=\bigoplus_{\sigma\in\mathcal{D}(\rho)}\sigma$ (\cite{bp12}, Theorem 13.8). 

For a finite unramified extension $F$ of $\mathbb{Q}_{p}$ of degree at least 3, Le uses a Diamond diagram attached to an 
\emph{irreducible} $\rho:\mathrm{Gal}(\overline{\mathbb{Q}_{p}}/F)\rightarrow GL_{2}(\overline{\mathbb{F}_{p}})$ to construct an infinite dimensional 
diagram which gives rise to a non-admissible smooth irreducible representation of $GL_{2}(F)$ (\cite{le19}). His strategy does not work for a Diamond diagram attached to an irreducible Galois representation of $\mathrm{Gal}(\overline{\mathbb{Q}_{p}}/\mathbb{Q}_{p^{2}})$ because such a diagram does not have suitable $\Pi$-action dynamics. However, for $F=\mathbb{Q}_{p^{2}}$, 
we observe that a Diamond diagram attached to a \emph{reducible split} $\rho$ has an indecomposable subdiagram with suitable $\Pi$-action dynamics so that  
Le's method can be used to obtain a non-admissible irreducible representation of $G = GL_{2}(\mathbb{Q}_{p^{2}})$.
\\

\noindent\textit{Acknowledgments}. We thank Daniel Le and Sandeep Varma for useful comments on earlier versions of this paper.
The second author also thanks Anand Chitrao for helpful discussions on diagrams.

\section{Reducible Diamond diagram}  
\indent Let $\omega_{2}$ be Serre's fundamental character of level 2 for the fixed embedding $\mathbb{F}_{p^{2}}\hookrightarrow\overline{\mathbb{F}_{p}}$, and 
let $\rho:\mathrm{Gal}(\overline{\mathbb{Q}_{p}}/\mathbb{Q}_{p^{2}})\rightarrow GL_{2}(\overline{\mathbb{F}_{p}})$ be a continuous reducible split generic 
Galois representation. The restriction of $\rho$ to the inertia subgroup is, up to a twist by some character, isomorphic to
\begin{equation*}
\begin{pmatrix}
\omega_{2}^{r_{0}+1+(r_{1}+1)p}&0\\0&1
\end{pmatrix}
\end{equation*} for some $0\leq r_{0},r_{1}\leq p-3$, not both equal to 0 or equal to $p-3$ (\cite{bre07}, Corollary 2.9 (i) and \cite{bp12}, Definition 11.7 (i)). Define the weight 
\begin{equation*}
\sigma:=(r_{0}+1, p-2-r_{1})\otimes\mathrm{det}^{p-1+r_{1}p}.
\end{equation*} 
Then the set of Diamond weights for $\rho$ is given by 
\begin{equation*}
\mathcal{D}(\rho)=\big\lbrace(r_{0}, r_{1}),\sigma,\sigma^{s},(p-3-r_{0},p-3-r_{1})\otimes\mathrm{det}^{r_{0}+1+(r_{1}+1)p}\big\rbrace
\end{equation*} (\cite{bp12}, Lemma 11.2 or Section 16, Example (ii)). Fix a Diamond diagram $(D_{0}(\rho),D_{1}(\rho),r)$ attached to $\rho$, 
and identify $D_{1}(\rho)$ with $D_{0}(\rho)^{I_{1}}$ as $IZ$-representations via $r$. There is a direct sum decomposition 
$D_{0}(\rho)=\bigoplus_{\nu\in\mathcal{D}(\rho)}D_{0,\nu}(\rho)$ of $K$-representations with $\mathrm{soc}_{K}D_{0,\nu}(\rho)=\nu$ (\cite{bp12}, Proposition 13.4).

Now define 
\begin{equation*}
D_{0}:=D_{0,\sigma}(\rho)\oplus D_{0,\sigma^{s}}(\rho)\hspace{.2cm}\text{and}\hspace{.2cm} D_{1}:=D_{0}^{I_{1}}.
\end{equation*} 
It follows from \cite{bp12}, Theorem 15.4 (ii) that $(D_{0},D_{1},r)$ is an indecomposable subdiagram of $(D_{0}(\rho),D_{1}(\rho),r)$. Set
\begin{equation*}
\tau:=(r_{0}+2, r_{1})\otimes\mathrm{det}^{p-2+(p-1)p} \hspace{.2cm} \text{and}\hspace{.2cm} \tau':=(p-1-r_{0},p-3-r_{1})\otimes\mathrm{det}^{r_{0}+(r_{1}+1)p}.
\end{equation*} The graded pieces of the 
socle filtrations of $D_{0,\sigma}(\rho)$ and $D_{0,\sigma^{s}}(\rho)$ are as follows (\cite{bp12}, Theorem 14.8 or Section 16, Example (ii)):
\begin{equation*}
D_{0,\sigma}(\rho):\begin{tikzcd}
\sigma\ar[r, no head]&\tau\oplus\tau^{s}\ar[r,no head]&(p-4-r_{0},r_{1}-1)\otimes\mathrm{det}^{r_{0}+2}
\end{tikzcd} 
\end{equation*}
\begin{equation*}
D_{0,\sigma^{s}}(\rho):\begin{tikzcd}
\sigma^{s}\ar[r, no head]&\tau'\oplus\tau'^{s}\ar[r,no head]&(r_{0}-1,p-4-r_{1})\otimes\mathrm{det}^{(r_{1}+2)p}.
\end{tikzcd}
\end{equation*}
We have from \cite{bp12}, Corollary 14.10 that \begin{equation}\label{d1decomposition}
D_{1}=\chi_{\sigma}\oplus\chi_{\tau}\oplus\chi_{\tau}^{s}\oplus\chi_{\sigma}^{s}\oplus\chi_{\tau'}\oplus\chi_{\tau'}^{s}. 
\end{equation}
For an $IZ$-representation $V$ and an $IZ$-character $\chi$, we write $V^{\chi}$ for the $\chi$-isotypic part of $V$.

\section{The infinite dimensional diagram and the construction}

Let $D_{0}(\infty):=\bigoplus_{i\in\mathbb{Z}}D_{0}(i)$ be the smooth $KZ$-representation with component-wise $KZ$-action, where there is a fixed 
isomorphism $D_{0}(i)\cong D_{0}$ of $KZ$-representations for every $i\in\mathbb{Z}$. Following \cite{le19}, we denote the natural inclusion 
$D_{0}\iso D_{0}(i)\hookrightarrow D_{0}(\infty)$ by $\iota_{i}$, and write $v_{i}:=\iota_{i}(v)$ for $v\in D_{0}$ for every $i\in\mathbb{Z}$. 
Let $D_{1}(\infty):=D_{0}(\infty)^{I_{1}}$. We define the $\Pi$-action on $D_{1}(\infty)$ as follows. 
Let $\lambda=(\lambda_{i})\in\prod_{i\in\mathbb{Z}}\overline{\mathbb{F}_{p}}^{\times}$. For all integers $i\in\mathbb{Z}$, define
\begin{equation*}
\Pi v_{i}:=
\begin{cases}
(\Pi v)_{i} &\text{if $v\in D_{1}^{\chi_{\sigma}}$,}\\
(\Pi v)_{i+1} &\text{if $v\in D_{1}^{\chi_{\tau}}$,}\\
\lambda_{i}(\Pi v)_{i} &\text{if $v\in D_{1}^{\chi_{\tau'}}$.}
\end{cases}
\end{equation*} 
This uniquely determines a smooth $N$-action on $D_{1}(\infty)$ such that $p=\Pi^{2}$ acts trivially on it. Thus we get a basic $0$-diagram 
$D(\lambda):=(D_{0}(\infty),D_{1}(\infty),\mathrm{can})$ with the above actions where can is the canonical inclusion $D_{1}(\infty)\hookrightarrow D_{0}(\infty)$.  

\begin{theorem}\label{existence}
There exists a smooth representation $\pi$ of $G$ such that 
\begin{enumerate}
\item $(\pi\big|_{KZ},\pi\big|_{N},\mathrm{id})$ contains $D(\lambda)$,
\item $\pi$ is generated by $D_{0}(\infty)$ as a $G$-representation, and
\item $\mathrm{soc}_{K}\pi=\mathrm{soc}_{K}D_{0}(\infty)$.
\end{enumerate}
\end{theorem}

\begin{proof}
Let $\Omega$ be the smooth injective $K$-envelope of $D_{0}$ equipped with the $KZ$-action such that $p$ acts trivially. The smooth injective $I$-envelope 
$\mathrm{inj}_{I}D_{1}$ of $D_{1}$ appears as an $I$-direct summand of $\Omega$. Let $e$ denote the projection of $\Omega$ onto $\mathrm{inj}_{I}D_{1}$. 
There is a unique $N$-action on $\mathrm{inj}_{I}D_{1}$ compatible with that of $I$ and compatible with the action of $N$ on $D_{1}$. By \cite{bp12}, Lemma 9.6, 
there is a non-canonical $N$-action on $(1-e)(\Omega)$ extending the given $I$-action. This gives an $N$-action on $\Omega$ whose restriction to $IZ$ is 
compatible with the action coming from $KZ$ on $\Omega$.

Now let $\Omega(\infty):=\bigoplus_{i\in\mathbb{Z}}\Omega(i)$ with component-wise $KZ$-action where there is a fixed isomorphism $\Omega(i)\cong\Omega$ of 
$KZ$-representations for every $i\in\mathbb{Z}$. We wish to define a compatible $N$-action on $\Omega(\infty)$. As before, denote the natural inclusion 
$\Omega\iso\Omega(i)\hookrightarrow \Omega(\infty)$ by $\iota_{i}$, and write $v_{i}:=\iota_{i}(v)$ for $v\in\Omega$. Let $\Omega_{\chi}$ denote the smooth 
injective $I$-envelope of an $I$-character $\chi$. Thus, from (\ref{d1decomposition}), we have $e(\Omega)=\mathrm{inj}_{I}D_{1}=\Omega_{\chi_{\sigma}}\oplus\Omega_{\chi_{\tau}}\oplus\Omega_{\chi_{\tau}^{s}}\oplus\Omega_{\chi_{\sigma}^{s}}\oplus\Omega_{\chi_{\tau'}}\oplus\Omega_{\chi_{\tau'}^{s}}$. 
If $v\in(1-e)(\Omega)$, we define $\Pi v_{i}:=(\Pi v)_{i}$ for all integers $i$. Otherwise, we define 
$\Pi v_{i}:=(\Pi v)_{i}$ if $v\in\Omega_{\chi_{\sigma}}$, $\Pi v_{i}:=(\Pi v)_{i+1}$ if $v\in\Omega_{\chi_{\tau}}$, 
and $\Pi v_{i}:=\lambda_{i}(\Pi v)_{i}$ if $v\in\Omega_{\chi_{\tau'}}$. By demanding that $\Pi^{2}$ acts trivially, this defines a 
smooth $N$-action on $\Omega(\infty)$ which is compatible with the $N$-action on $D_{1}(\infty)$, and whose restriction to $IZ$ is compatible with the action 
coming from $KZ$ on $\Omega(\infty)$. By \cite{pas04}, Corollary 5.5.5, we have a smooth $G$-action on $\Omega(\infty)$. We then take $\pi$ to be the $G$-representation 
generated by $D_{0}(\infty)$ inside $\Omega(\infty)$. If follows easily from the construction that $\pi$ satisfies the properties (1), (2) and (3). 
\end{proof}

\begin{theorem}
If $\lambda_{i}\neq\lambda_{0}$ for all $i\neq 0$, then any smooth representation $\pi$ of $G$ satisfying the properties (1), (2), and (3) of Theorem \ref{existence} is irreducible and non-admissible.
\end{theorem}
\begin{proof} 
Let $\pi'\subseteq\pi$ be a non-zero subrepresentation of $G$. By property (3), we have either $\mathrm{Hom}_{K}(\sigma,\pi')\neq 0$ or $\mathrm{Hom}_{K}(\sigma^{s},\pi')\neq 0$. We consider the case $\mathrm{Hom}_{K}(\sigma,\pi')\neq 0$; the other case is treated analogously. There exists a non-zero $(c_{i})\in\bigoplus_{i\in\mathbb{Z}}\overline{\mathbb{F}_{p}}$ such that 
\begin{equation*}
  \big(\sum_{i}c_{i}\iota_{i}\big)(D_{0,\sigma}(\rho))\cap\pi'\neq 0.
\end{equation*} 
We claim that 
\begin{equation}\label{claim}
  \big(\sum_{i}c_{i}\iota_{i+j}\big)(D_{0})\subset\pi' \hspace{.3cm} \text{for all}  \hspace{.1cm} j\in\mathbb{Z}.
\end{equation} 	
Note that the $K$-socle of $\big(\sum_{i}c_{i}\iota_{i}\big)(D_{0,\sigma}(\rho))$ is $\big(\sum_{i}c_{i}\iota_{i}\big)(\sigma)$ which is irreducible. Hence, $\big(\sum_{i}c_{i}\iota_{i}\big)(D_{0,\sigma}(\rho))\cap\pi'\neq 0$ implies that $\big(\sum_{i}c_{i}\iota_{i}\big)(\sigma)\subset\pi'$. The map $\delta$ defined in \cite{bp12}, Section 15 takes $\sigma$ to $\sigma^{s}$ and vice versa. Therefore, the arguments in the proof of \cite{bp12}, Theorem 19.10 (i) and Lemma 19.7 imply that 
\begin{equation*}
  \big(\sum_{i}c_{i}\iota_{i}\big)(D_{0,\delta(\sigma)}(\rho)) = \big(\sum_{i}c_{i}\iota_{i}\big)(D_{0,\sigma^{s}}(\rho))\subset\pi'.
\end{equation*} 
Repeating the argument now for $\sigma^{s}$, we see that $\big(\sum_{i}c_{i}\iota_{i}\big)(D_{0,\sigma}(\rho))\subset\pi'$. Thus,
\begin{equation*}
\big(\sum_{i}c_{i}\iota_{i}\big)(D_{0})\subset\pi'.
\end{equation*}
Therefore,
\begin{equation*}
\big(\sum_{i}c_{i}\iota_{i}\big)(D_{1}^{\chi_{\tau}})\subset\pi'\hspace{.2cm} \text{and}\hspace{.2cm} \big(\sum_{i}c_{i}\iota_{i}\big)(D_{1}^{\chi_{\tau}^{s}})\subset\pi'.
\end{equation*} 
Since $\pi'$ is stable under the $\Pi$-action, we have 
\begin{equation*}
\big(\sum_{i}c_{i}\iota_{i+1}\big)(D_{1}^{\chi_{\tau}^{s}})\subset\pi'\hspace{.2cm} \text{and}\hspace{.2cm} \big(\sum_{i}c_{i}\iota_{i-1}\big)(D_{1}^{\chi_{\tau}})\subset\pi'.
\end{equation*}
In particular,
\begin{equation*}
\big(\sum_{i}c_{i}\iota_{i+1}\big)(D_{0,\sigma}(\rho))\cap\pi'\neq 0
\hspace{.2cm} \text{and}\hspace{.2cm} \big(\sum_{i}c_{i}\iota_{i-1}\big)(D_{0,\sigma}(\rho))\cap\pi'\neq 0.
\end{equation*}
By the same arguments as above, we find that 
\begin{equation*}
\big(\sum_{i}c_{i}\iota_{i+1}\big)(D_{0})\subset\pi'\hspace{.2cm} \text{and}\hspace{.2cm} \big(\sum_{i}c_{i}\iota_{i-1}\big)(D_{0})\subset\pi'.
\end{equation*}
The claim is now proved by repeatedly using the $\Pi$-action.

For $(d_{i})\in\bigoplus_{i\in\mathbb{Z}}\overline{\mathbb{F}_{p}}$, let $\#(d_{i})$ denote the number of non-zero $d_{i}$'s. Among all the non-zero elements 
 $(c_{i})$ of $\bigoplus_{i\in\mathbb{Z}}\overline{\mathbb{F}_{p}}$ for which $\big(\sum_{i}c_{i}\iota_{i}\big)(D_{0})\subset\pi'$, we pick one with $\#(c_{i})$ minimal. 
We may also assume that $c_{0}\neq 0$ using (\ref{claim}). We now show that $\#(c_{i})=1$. Assume to the contrary that $\#(c_{i})>1$. 
Since $\big(\sum_{i} c_{i}\iota_{i}\big)(D_{1}^{\chi_{\tau'}}) \subset\pi'$ and
$\pi'$ is stable under the $\Pi$-action, we have 
\begin{equation*}
\big(\sum_{i}\lambda_{i} c_{i}\iota_{i}\big)(D_{1}^{\chi_{\tau'}^{s}})\subset\pi'.
\end{equation*} 
Since $\big(\sum_{i}\lambda_{0} c_{i}\iota_{i}\big)(D_{1}^{\chi_{\tau'}^{s}})$ is also clearly in $\pi'$, subtracting it from the above, we get 
\begin{equation*}
\big(\sum_{i}(\lambda_{i}-\lambda_{0})c_{i}\iota_{i}\big)(D_{1}^{\chi_{\tau'}^{s}})\subset\pi'.
\end{equation*}
Writing $(c_{i}'):=((\lambda_{i}-\lambda_{0})c_{i})$, we see that 
\begin{equation*}
\big(\sum_{i}c_{i}'\iota_{i}\big)(D_{0,\sigma^{s}}(\rho))\cap\pi'\neq 0.
\end{equation*}
Following the same arguments as in the previous paragraph, we get that $\big(\sum_{i}c_{i}'\iota_{i}\big)(D_{0})\subset\pi'$. 
However, the hypothesis $\lambda_{i}\neq\lambda_{0}$ for all $i\neq 0$, and the assumption $\#(c_{i})>1$ imply that $(c_{i}')$ is non-zero and 
$\#(c_{i}')=\#(c_{i})-1$ contradicting the minimality of $\#(c_{i})$. Therefore, we have $c_{0}\iota_{0}(D_{0})\subset\pi'$. 
So $\iota_{0}(D_{0})\subset\pi'$. Using (\ref{claim}) again, we get that $\bigoplus_{j\in\mathbb{Z}}\iota_{j}(D_{0})=D_{0}(\infty)\subset\pi'$. 
By property (2), we have $\pi'=\pi$.   

Non-admissibility of $\pi$ is clear because $\pi^{K_{1}}\supseteq\mathrm{soc}_{K}\pi$ and $\mathrm{soc}_{K}\pi$ is not finite dimensional by property (3).
\end{proof}

\end{document}